\documentclass[reqno]{amsart}

\input xy
\xyoption{all}
\usepackage{accents}
\usepackage{epsfig}
\usepackage{color}
\usepackage{amsthm}
\usepackage{amssymb}
\usepackage{amsmath}
\usepackage{amscd}
\usepackage{amsopn}
\usepackage{graphicx}

\usepackage{xspace}

\usepackage{url}
\usepackage{enumitem, hyperref}\hypersetup{colorlinks}

\usepackage{color} 

\definecolor{darkred}{rgb}{1,0,0} 
\definecolor{darkgreen}{rgb}{0,0.8,0}
\definecolor{darkblue}{rgb}{0,0,1}

\hypersetup{colorlinks,
linkcolor=darkblue,
filecolor=darkgreen,
urlcolor=darkred,
citecolor=darkgreen}

\makeatletter
\def\reflb#1#2{\begingroup
    #2%
    \def\@currentlabel{#2}%
    \phantomsection\label{#1}\endgroup
}
\makeatother

\numberwithin{equation}{section}
\newtheorem {Theorem}{Theorem}
\numberwithin{Theorem}{section}

\newtheorem {Corollary}[Theorem]{Corollary}
\theoremstyle{definition}

\theoremstyle{remark}
\newtheorem{Remark}[Theorem]{Remark}
\newtheorem{Example}[Theorem]{Example}

\def    \eps    {\epsilon}

\newcommand{\FF}{{\mathcal F}}

\def    \R      {{\mathbb R}}

\def    \12    {{\frac{1}{2}}}

\def    \p      {\partial}

\def    \Lk  {\operatorname{Lk}}

\begin{document}


\setlength{\smallskipamount}{6pt}
\setlength{\medskipamount}{10pt}
\setlength{\bigskipamount}{16pt}





\title[Unique Ergodicity and the
  Contact Type condition]{A Remark on Unique Ergodicity and the
  Contact Type condition}

\author[Viktor Ginzburg]{Viktor L. Ginzburg}
\author[C\'esar Niche]{C\'esar J. Niche}

\address{VG: Department of Mathematics, UC Santa Cruz, Santa Cruz, CA
  95064, USA} \email{ginzburg@ucsc.edu}

\address{CN: 
Departamento de Matem\'atica Aplicada,
Universidade Federal do Rio de Janeiro,
Rio de Janeiro, RJ - CEP 21941-909,  Brasil} \email{cniche@im.ufrj.br}



\date{\today} 

\thanks{The work is partially supported by NSF grant 
  DMS-1308501}

\bigskip

\begin{abstract}
We prove that for a broad class of exact symplectic manifolds including
$\R^{2m}$ the Hamiltonian flow on a regular compact
energy level of an autonomous Hamiltonian cannot be
uniquely ergodic. This is a consequence of the Weinstein
conjecture and an observation that a Hamiltonian structure with
non-vanishing self-linking number must have contact type. We apply
these results to show that certain types of exact twisted geodesic
flows cannot be uniquely ergodic.
\end{abstract}

\maketitle


\section{Results}
\label{sec:results}

\subsection{Introduction}
\label{sec:intro}

In this paper we show that the Hamiltonian flow on a regular compact
energy level of an autonomous Hamiltonian on $\R^{2m}$ cannot be
uniquely ergodic. In fact, the result holds for a much larger class of
symplectic manifolds. For instance, it is sufficient to assume that
every compact subset of the ambient manifold has finite
Hofer--Zehnder capacity.

To put these results in perspective, recall that by the so-called
almost existence theorem proved by Hofer and Zehnder, \cite{HZ}, and
by Struwe, \cite{St}, almost all, in the sense of measure theory,
regular energy levels of a proper Hamiltonian on $\R^{2m}$ carry
periodic orbits. (Again, this theorem holds for a much broader class
of ambient manifolds; see, e.g., \cite{HZ} and \cite{Gi:We} for a
survey of related results.)  On the other hand, it is also known that
periodic orbits need not exist on all regular energy levels. To be
more specific, when $2m\geq 4$, there exists a proper Hamiltonian
$H\colon \R^{2m}\to \R$ with only one critical point, which is thus a
minimum, carrying no periodic orbits of its Hamiltonian flow on a
regular level.  Hamiltonians with this property are known as
counterexamples to the Hamiltonian Seifert conjecture. (The
Hamiltonian $H$ is $C^\infty$ smooth when $2m>4$. However, when
$2m=4$, the Hamiltonian is only $C^2$ and it is not known if a
$C^\infty$-smooth $H$ exists.) We refer the reader to the survey
\cite{Gi:Seifert} and to \cite{GG:Seifert} for further references and
a more detailed discussion of the counterexamples to the Hamiltonian
Seifert conjecture.)

This naturally leads to the question of how far from having periodic
orbits can the flow on a regular level of $H$ can be. Unique
ergodicity is arguably the most extreme form of ``aperiodicity'', and
our results show that the flow on a regular level of a proper
Hamiltonian on $\R^{2m}$ cannot be so ``aperiodic''. This is also true
for many other, but not all, ambient symplectic manifolds; see
Corollary \ref{cor:ae-gen} and Example \ref{ex:horo}.

As an application of this approach, we show that in certain situations
exact twisted geodesic flows cannot be uniquely ergodic; see Corollary
\ref{cor:magnetic}.

The proof of these results is rather indirect, although it amounts to
a simple combination of well known facts. First, we use the contact
type criterion due to McDuff, \cite{McD}, to show that a uniquely
ergodic ``Hamiltonian structure'' meeting a certain natural
requirement must have contact type. This is our Theorem
\ref{thm:main}, which is also closely related to a result of Taubes
from \cite{Ta:UE}; see Corollary \ref{cor:Tau}.  The requirement is
that the self-linking number of the structure is non-zero and it is
automatically satisfied for a closed hypersurface bounding a domain in
an exact symplectic manifold. Next, we observe that whenever the
Weinstein conjecture is established for the resulting class of contact
type hypersurfaces, the Reeb flow has a closed characteristic, and
hence cannot be uniquely ergodic; see Corollaries \ref{cor:ae-r2n} and
\ref{cor:ae-gen}. Another way to interpret these results is that
certain natural, essentially topological, requirements on a
Hamiltonian structure imply a constraint on its dynamics, namely, that
the structure cannot be uniquely ergodic.

\subsection{Main Results}
\label{sec:results}
Before stating the main results of this note, let us introduce some
terminology.

Let $M^{2n+1}$ be a closed oriented manifold and $\omega$ be a
maximally non-degenerate (i.e., $\omega^n\neq 0$ anywhere on $M$)
exact two-form on $M$. For brevity, we will refer to $\omega$ as a
\emph{Hamiltonian structure}, cf.\ \cite{EKP}. (We emphasize that we
do not impose any further conditions, e.g., stability, on $\omega$.)
An example of a Hamiltonian structure is the restriction of a
symplectic form to a compact regular level of a Hamiltonian when the
symplectic form on the level is exact. Note that for a fixed volume
form $\mu$ on $M$ the non-vanishing vector field $X$ determined by the
condition $i_X\mu=\omega^n$ is necessarily volume
preserving. Conversely, in dimension three, once $M$ and $\mu$ are
fixed, $X\mapsto i_X\mu$ is a one-to-one correspondence between exact
divergence-free vector fields and Hamiltonian structures. (In higher
dimensions, a non-vanishing divergence-free vector field $X$ such that
$i_X\mu$ is exact need not come from a Hamiltonian structure.)  In any
event, a Hamiltonian structure gives rise to the line field
$\ker \omega$ which integrates to a transversely symplectic
one-dimensional foliation $\FF$ called the \emph{characteristic
  foliation} of $\omega$. The leaves of $\FF$ are the integral curves
of $X$. Recall that a flow is said to be \emph{uniquely ergodic} if it
admits only one ergodic measure. For volume preserving flows this must
then be the volume form. The flow is \emph{minimal} if its every orbit
is dense. (See, e.g., \cite{KH} for further details.)  We call
$\omega$ uniquely ergodic (minimal) when $\FF$, i.e., the flow of $X$,
is uniquely ergodic (minimal). Finally, we say that $\omega$ has
\emph{contact type} when $\omega$ has a contact primitive, i.e., there
exists a one-form $\alpha$ such that $d\alpha=\omega$ and
$\alpha\wedge(d\alpha)^n\neq 0$ everywhere.

To a Hamiltonian structure $\omega$ we associate a number 
$$
\Lk(\omega)=\int_M\alpha\wedge\omega^n,
$$
where $\alpha$ is a primitive of $\omega$. Since $M$ is closed,
$\Lk(\omega)$ is well defined, i.e., independent of $\alpha$, by the
Stokes' formula. (The sign of $\Lk(\omega)$ depends on the orientation
of $M$.) Clearly, $\Lk(\omega)\neq 0$ when $\omega$ has contact
type. In dimension three, $\Lk(\omega)$ can be interpreted as the
self-linking number of $\FF$ or the asymptotic Hopf invariant;
see~\cite{Ar:ah}.

Our key technical result is the following theorem.

\begin{Theorem}
\label{thm:main}
Let $\omega$ be a uniquely ergodic Hamiltonian structure with
$\Lk(\omega)\neq 0$. Then $\omega$ has contact type.
\end{Theorem}

The proof of this theorem, given in detail in Section \ref{sec:proof},
is an easy application of the contact type criterion from
\cite{McD}.  The theorem is essentially a negative result saying that
a Hamiltonian structure with $\Lk(\omega)\neq 0$ can not be uniquely
ergodic if we assume the Weinstein conjecture to hold. Indeed, the Weinstein
conjecture asserts that the characteristic foliation of a contact type
Hamiltonian structure has a periodic orbit, and hence cannot be
uniquely ergodic. (See, e.g., \cite{Hu} for a detailed
discussion). Thus this is really so in the cases where the Weinstein
conjecture has been established, e.g., when $M$ is three-dimensional,
\cite{Ta:We}, or when $M$ is a hypersurface in $\R^{2m}$, \cite{Vi:We},
or even more generally $M$ is a displaceable hypersurface in a
sufficiently nice symplectic manifold; see, e.g., \cite{Gi:We,Hu,HZ}
for further references. Moreover, in some instances, the condition
that $\Lk(\omega)\neq 0$ is satisfied automatically. For example, we
have

\begin{Corollary}
\label{cor:ae-r2n}
Let $M$ be a connected closed hypersurface in $\R^{2m}$. Then the characteristic
foliation on $M$ cannot be uniquely ergodic. Equivalently, an autonomous
Hamiltonian flow on a connected regular level of a proper Hamiltonian on
$\R^{2m}$ cannot be uniquely ergodic.
\end{Corollary}

\begin{proof} Let $W$ be the domain bounded by $M$ in $\R^{2m}$ and
  let $M=\p W$ be oriented by the outward normal. Let us
  also denote by $\omega$ the standard symplectic form on $\R^{2m}$.
  The result is then an immediate consequence of Theorem \ref{thm:main},
  the Weinstein conjecture for hypersurfaces in $\R^{2m}$ mentioned
  above, and  the equality
$$
\Lk(\omega) =\int_W\omega^m>0,
$$
which in turn readily follows from the Stokes' formula.
\end{proof}

This simple observation deserves a further discussion. First, we note
for future reference that a much more general result holds.

\begin{Corollary}
\label{cor:ae-gen}
Let $M$ be a connected closed hypersurface in an exact geometrically
bounded symplectic manifold $V$ bounding a domain $W$ with finite
Hofer--Zehnder capacity. Then the characteristic foliation on $M$
cannot be uniquely ergodic. In particular, if every compact set in $V$
has finite Hofer--Zehnder capacity (e.g., is displaceable), an autonomous
Hamiltonian flow on a connected regular level of a proper, bounded
from below Hamiltonian on $V$ cannot be uniquely ergodic.
\end{Corollary}

We refer the reader to, e.g., \cite{Gi:We,HZ} for the definitions and
notions used in the corollary, noting here only that the conditions on
$V$ and $W$ are automatically satisfied when $V$ is a subcritical
symplectic Stein manifold. The proof of this corollary is identical to
the proof of Corollary \ref{cor:ae-r2n}, but now we use the fact that
the Weinstein conjecture holds for $M$ since $W$ has finite
Hofer--Zehnder capacity; see \cite{HZ}. The corollary can be further
generalized; however, this variant is more than sufficient for our
purposes.

Next, recall that a uniquely ergodic volume-preserving flow is
automatically minimal; see, e.g., \cite{KH}. The converse is not true;
see \cite{Fu} and also \cite{KH} for further details. (However, the
authors are not aware of any example of a uniquely ergodic, but not
minimal, Hamiltonian structure $\omega$ with $\Lk(\omega)\neq 0$.) It
is tempting to conjecture that in the setting of Theorem
\ref{thm:main}, or at least in the context of Corollaries
\ref{cor:ae-r2n} and \ref{cor:ae-gen}, the Hamiltonian structure
cannot be minimal. This would be a much deeper and more difficult
result than Theorem \ref{thm:main}. For hypersurfaces in $\R^4$, a
proof of this fact was recently announced by Fish and Hofer,
\cite{FH}.

Turning to another application of Theorem \ref{thm:main}, we have

\begin{Corollary}[Taubes, \cite {Ta:UE}]
\label{cor:Tau}
Let $\omega$ be a Hamiltonian structure on a closed (oriented) 3-manifold
$M$ with $\Lk(\omega)\neq 0$. Then $\omega$ is not uniquely ergodic.
\end{Corollary}

This corollary is an immediate consequence of Theorem \ref{thm:main}
and, on the non-trivial side, the Weinstein conjecture for 3-manifolds
proved by Taubes in \cite{Ta:We}; see also \cite{Hu} for a survey of
related results. Here the new point is that to establish the corollary
we use the assertion of the Weinstein conjecture rather than its proof
as in \cite{Ta:UE}.

\subsection{Applications and Examples: Twisted Geodesic Flows}
\label{sec:magn}
Let $\sigma$ be a closed 2-form (a \emph{magnetic field}) on a closed
Riemannian manifold $B$. We equip $T^*B$ with the twisted symplectic
structure $\omega=\omega_0+\pi^*\sigma$, where $\omega_0$ is the
standard symplectic form on $T^*B$ and $\pi\colon T^*M\to B$ is the
natural projection, and let $K$ be the standard kinetic energy
Hamiltonian on $T^*B$ arising from the Riemannian metric on $B$.  The
Hamiltonian flow of $K$ on $T^*B$ governs the motion of a charge on
$B$ in the magnetic field $\sigma$ and is referred to as a
\emph{twisted geodesic flow}. In contrast with the geodesic flow (the
case $\sigma=0$), the dynamics of a twisted geodesic flow on an energy
level $M_\eps=\{K=\eps^2/2\}$ depends on the level.

\begin{Example}[Horocycle flow] 
\label{ex:horo}
Let $B$ be a closed surface equipped
  with a metric of constant negative curvature $-1$ and let $\sigma$
  be the area form on $B$. Note that the restriction
  $\omega|_{M_\eps}$ is exact for every $\eps>0$ although the form
  $\omega$ is not exact on $T^*B$.  When $0<\eps<1$, every orbit of
  the Hamiltonian flow on $M_\eps$ is closed and all orbits have the
  same period. When $\eps>1$, the flow on $M$ is smoothly conjugate to
  the geodesic flow, up to a time change. The flow on $M_1$ is the
  so-called horocycle flow. (This observation goes back to
  \cite{Ar61}.)  The horocycle flow is known to be uniquely ergodic,
  \cite{Fu:UE}, and, as is easy to see, $\Lk(\omega|_{M_1})=0$, which
  shows that the conditions that $\Lk(\omega)\neq 0$ in Theorem
  \ref{thm:main} and that $V$ is exact in Corollary \ref{cor:ae-gen}
  are essential.
\end{Example}

\begin{Corollary}
\label{cor:magnetic}
Assume that the form $\sigma$ is exact. Then, for $\eps>0$
sufficiently small, the Hamiltonian flow on $M_\eps$ cannot be uniquely
ergodic.
\end{Corollary}

This corollary is an immediate consequence of Corollary
\ref{cor:ae-gen} and a theorem of Schlenk, \cite{Sc}, asserting that a small
neighborhood of the zero section in $T^*B$ has finite Hofer-Zehnder
capacity.

\begin{Remark}
  Corollary \ref{cor:magnetic} can also be generalized in a variety of
  ways. For instance, one can replace the level $M_\eps$ of $K$ by any
  closed hypersurface in a sufficiently small neighborhood of the zero
  section. Furthermore, $T^*B$ can be replaced by any exact
  geometrically bounded symplectic manifold, meeting some minor
  additional requirements, and the zero section can be replaced by any
  submanifold $B$ such that $\omega|_B\neq 0$, cf.\ \cite{Gu}.
\end{Remark}

\section{Proof of Theorem \ref{thm:main}}
\label{sec:proof}
As has been mentioned in Section \ref{sec:results}, the theorem readily
follows from McDuff's contact type criterion, \cite{McD}, based, in
turn, on a work of Sullivan, \cite{Su}. To state this criterion, we
need first to introduce several notions.

Fix a volume form $\mu$ on $M$ and let $X$ be the vector field
uniquely determined by the condition $i_X\mu=\omega^n$. The integral
curves of the flow of $X$ are parametrized characteristics of
$\omega$.  Consider the currents of the form $X\otimes \nu$, where
$\nu$ is an invariant measure on $M$, acting on a 1-form $\alpha$ as
$$
\left<X\otimes \nu,\alpha\right>:=\int_M\alpha(X)\,d\nu,
$$
By definition, such a current is a
\emph{structure boundary} if
$$
\left<X\otimes \nu,\beta\right>=0
$$
whenever $\beta$ is closed; see \cite{McD,Su}. For instance, a
contractible periodic orbit of $X$, i.e., a closed characteristic of
$\omega$, gives rise to a structure boundary which is simply the
integral over the orbit. 

Let $\alpha$ be a primitive of $\omega$. To a structure boundary
we can associate its ``action'' 
$$
A(X\otimes \nu):=\left<X\otimes \nu,\alpha\right>,
$$
which is clearly independent of the choice of $\alpha$. McDuff's
contact type criterion asserts that $\omega$ has contact type if and
only if
$$
A(X\otimes \nu)\neq 0
$$
for all structure boundaries $X\otimes \nu$; see \cite{McD}. (The
observation that the actions on all contractible closed
characteristics must have the same sign for a closed contact type
hypersurface goes back to \cite{We}, where it is used to construct
hypersurfaces in $\R^{2m}$ which do not have contact type.)

In the setting of the theorem, $\mu$ is the only invariant measure
since $\omega$ (and hence $X$) are uniquely ergodic. Thus $X\otimes
\mu$ is the only candidate for a structure boundary and, in fact, it
is a structure boundary. Indeed, assume that $d\beta=0$. Then, since
$i_X\mu=\omega^n$ and $d\alpha=\omega$, we have
\begin{eqnarray*}
\left<X\otimes \mu,\beta\right> &=&\int_M\beta(X)\,d\mu\\
&=&-\int_M\beta\wedge(d\alpha)^n\\
&=&\int_Md[\beta\wedge\alpha\wedge(d\alpha)^{n-1}]\\
&=&0
\end{eqnarray*}
by Stokes' formula. (Alternatively, one can argue that $X\otimes \mu$
must be a structure boundary since structure boundaries always exists; see
\cite{Su}.)

In a similar vein, we have
$$
A(X\otimes
\mu)=\int_M\alpha(X)\,d\mu=-\int_M\alpha\wedge\omega^n=-\Lk(\omega)\neq 0
$$
by the hypotheses of the theorem. Hence, the conditions of McDuff's
criterion are met and $(M,\omega)$ has contact type.
\hfill\qed

\medskip
\noindent{\bf Acknowledgements.} The authors are grateful to Ba\c sak
G\"urel for useful discussions.  A part of this work was carried out
while the first author was visiting the ICMAT, Madrid, and he would
like to thank the ICMAT for its warm hospitality and support.


\begin{thebibliography}{CKRTZ}

\bibitem[Ar61]{Ar61} V.I. Arnold,
Some remarks on flows of line elements and frames, (Russian)
\emph{Dokl.\ Akad.\ Nauk SSSR}, \textbf{138} (1961), 255--257. 

\bibitem[Ar86]{Ar:ah} V.I. Arnold, The asymptotic Hopf invariant and its
  applications, \emph{Selecta Math.\ Soviet.}, \textbf{5} (1986),
327--345.

\bibitem[EKP]{EKP} Ya. Eliashberg, S.S. Kim, L. Polterovich, Geometry
  of contact transformations and domains: orderability versus
  squeezing, \emph{Geom.\ Topol.}, \textbf{10} (2006), 1635--1747.



\bibitem[FH]{FH}
J. Fish, H. Hofer, work in progress.


\bibitem[Fu61]{Fu}
H. Furstenberg,
Strict ergodicity and transformation of the torus, 
\emph{Amer.\ J. Math.}, \textbf{83} (1961), 573--601. 

\bibitem[Fu73]{Fu:UE} H. Furstenberg, The unique ergodicity of the
  horocycle flow, in \emph{Recent advances in topological dynamics
  (Proc.\ Conf., Yale Univ., New Haven, Conn., 1972; in honor of Gustav
  Arnold Hedlund)}, pp.\ 95--115. Lecture Notes in Math., Vol.\ 318,
  Springer, Berlin, 1973.

\bibitem[Gi99]{Gi:Seifert} 
V.L. Ginzburg, Hamiltonian dynamical
  systems without periodic orbits, in \emph{Northern California
    Symplectic Geometry Seminar}, 35--48,
    Amer.\ Math.\ Soc.\ Transl.\ Ser.\ 2, 196, Amer.\ Math.\ Soc.,
    Providence, RI, 1999.

\bibitem[Gi05]{Gi:We}
V.L. Ginzburg, 
The Weinstein conjecture and theorems of nearby and almost existence,
in \emph{The Breadth of Symplectic and Poisson Geometry}, 139--172, 
Progr.\ Math., \textbf{232}, Birkh\"auser Boston, Boston, MA, 2005. 


\bibitem[GG]{GG:Seifert}
V.L.  Ginzburg, B.Z. G\"urel, A $C^2$-smooth counterexample to the
Hamiltonian Seifert conjecture in $\R^4$, \emph{Ann.\ of Math. (2)},
\textbf{158} (2003), 
953--976.


\bibitem[G\"u]{Gu}
B.Z. G\"urel,
Totally non-coisotropic displacement and its applications to
Hamiltonian dynamics, \emph{Comm.\ Contemp.\ Math.}, \textbf{10}
(2008), 
1103--1128.

\bibitem[HZ]{HZ}
H. Hofer, E. Zehnder,
\emph{Symplectic Invariants and Hamiltonian Dynamics}, Birk\"auser
Verlag, Basel, 1994.

\bibitem[Hu]{Hu} M. Hutchings, Taubes's proof of the Weinstein
conjecture in dimension three, \emph{Bull.\ Amer.\ Math.\ Soc.\ (N.S.)}, \textbf{47} (2010),
73--125.

\bibitem[KH]{KH} 
A. Katok, B. Hasselblatt, 
\emph{Introduction to the Modern Theory of Dynamical Systems},
Encyc.\ of Mathematics and its Applications, \textbf{54}, Cambridge
Univ.\ Press, Cambridge, 1995.


\bibitem[McD]{McD} D. McDuff, Applications of convex integration to
symplectic and contact geometry, \emph{Ann.\ Inst.\ Fourier (Grenoble)},
\textbf{37} (1987), 
107--133.


\bibitem[Sc]{Sc} F. Schlenk, Applications of Hofer’s geometry to
Hamiltonian dynamics, \emph{Comment.\ Math.\ Helv.}, \textbf{81}
(2006), 105--121.

\bibitem[St]{St}
M. Struwe, Existence of periodic solutions of Hamiltonian
systems on almost every energy surfaces, \emph{Bol.\ Soc.\ Bras.\ Mat.},
\textbf{20} (1990), 49--58.


\bibitem[Su]{Su} D. Sullivan, Cycles for the dynamical study of
  foliated manifolds and complex manifolds, \emph{Invent.\ Math.},
  \textbf{36} (1976), 225--255.

\bibitem[Ta07]{Ta:We}
C.H. Taubes, 
The Seiberg--Witten equations and the Weinstein conjecture,
\emph{Geom.\ Topol.}, \textbf{11} (2007), 2117--2202. 

\bibitem[Ta09]{Ta:UE}
C.H. Taubes, An observation concerning uniquely ergodic vector fields
on 3-manifolds,  \emph{J. G\"okova Geom.\ Topol.\ GGT}, \textbf{3} (2009), 9--21. 

\bibitem[Vi]{Vi:We}
C. Viterbo, A proof of Weinstein's conjecture in $\R^{2n}$,
\emph{Ann.\ Inst.\ H. Poincar\'e Anal.\ Non Lin\'eaire}, \textbf{4}
(1987), 
337--356.


\bibitem[We]{We}
A. Weinstein, 
On the hypotheses of Rabinowitz' periodic orbit theorems,
\emph{J. Differential Equations}, \textbf{33} (1979), 
353--358. 

\end{thebibliography}
\end{document}